\documentclass[11pt,twoside, a4paper]{amsart}

\title
{           {\protect\hfill \normalfont \tiny
            \\ \vspace{10pt}}
Open descent and strong approximation
}

\author{Dasheng Wei
}

\date{\today}


\usepackage[arrow,matrix]{xy}
\usepackage{mathptmx}
\usepackage{amscd}
\usepackage{amssymb}
\usepackage{eucal}
\usepackage{amsxtra}
\usepackage{verbatim}
\usepackage{url}
\usepackage{enumerate}
\usepackage[english]{babel}

\DeclareFontEncoding{OT2}{}{} 
\DeclareTextFontCommand{\textcyr}{\fontencoding{OT2}
    \fontfamily{wncyr}\fontseries{m}\fontshape{n}\selectfont}

\usepackage{comment}

\theoremstyle{plain}

\newtheorem{theorem}{Theorem}

\newtheorem{lemma}[theorem]{Lemma}
\newtheorem{corollary}[theorem]{Corollary}

\newtheorem{conditional-result}[theorem]{Conditional Result}

\newtheorem{theorem?}{Theorem(?)} [section]
\newtheorem{proposition?}[theorem]{Proposition(?)}
\newtheorem{lemma?}[theorem]{Lemma(?)}
\newtheorem{corollary?}[theorem]{Corollary(?)}

\newtheorem*{theorem*}{Theorem}
\newtheorem*{proposition*}{Proposition}
\newtheorem*{lemma*}{Lemma}
\newtheorem*{corollary*}{Corollary}
\newtheorem*{question*}{Question}
\newtheorem*{conjecture*}{Conjecture}
\newtheorem*{claim*}{Claim}

\newtheorem*{introtheorem*}{Theorem}
\newtheorem*{introproposition*}{Proposition}
\newtheorem*{introlemma*}{Lemma}
\newtheorem*{introcorollary*}{Corollary}

\theoremstyle{definition}

\newtheorem{example}[theorem]{Example}

\newtheorem*{definition*}{Definition}
\newtheorem*{example*}{Example}

\theoremstyle{remark}

\newtheorem*{remark*}{Remark}

\numberwithin{equation}{section}
\numberwithin{theorem}{section}

\DeclareSymbolFont{rsfs}{U}{rsfs}{m}{n}
\DeclareSymbolFontAlphabet{\mathcal}{rsfs}

\DeclareFontEncoding{OT2}{}{} 
\DeclareTextFontCommand{\textcyr}{\fontencoding{OT2}
    \fontfamily{wncyr}\fontseries{m}\fontshape{n}\selectfont}
\newcommand{\Sh}{\textcyr{Sh}}

\newcommand{\ZZ}{{\mathbb{Z}}}
\newcommand{\QQ}{{\mathbb{Q}}}

\newcommand{\Gal}{{\rm Gal}}
\newcommand{\Pic}{{\rm Pic}}
\newcommand{\Br}{{\rm Br}}
\renewcommand{\ker}{{\rm ker}}

\newcommand{\Hom}{{\rm Hom}}

\newcommand{\kbar}{{\overline{k}}}

\newcommand{\Ybar}{{\overline Y}}

\newcommand{\Deltabar}{{\overline \Delta}}

\newcommand{\That}{{\widehat{T}}}

\newcommand{\Mhat}{{\widehat{M}}}

\newcommand{\Xbar}{{\overline{X}}}

\def\G{{\mathbb{G}}}

\def\Z{{\ZZ}}
\def\Q{{\QQ}}

\def\Div{{\textup{Div}}}

\newcommand{\Tbar}{{\overline T}}

\def\AA{{\mathcal{A}}}

\def\T{{\mathcal{T}}}
\def\A{\mathbf{A}}

\begin{document}

\begin{abstract} We give a new version of the open descent theory of Harari and Skorobogatov. As an application of the new version, we  prove that  some algebraic varieties satisfy strong approximation.

\end{abstract}

\maketitle

\section{Introduction}

Let $X$ be a smooth and geometrically integral variety
over a number field $k$.
The descent theory of Colliot-Th\'el\`ene and
Sansuc (\cite {CTS87})
describes arithmetic properties of $X$ in terms of
$X$-torsors under $k$-groups of multiplicative type.
It interprets
the {\it Brauer--Manin obstruction} to the existence of
a rational point (or to weak approximation) on $X$
in terms of the obstructions defined by torsors.

Let $\kbar$ be an algebraic closure of $k$.
The first applications of the descent theory was stated in
\cite {CTS87} for geometrically rational proper varieties; in this case it is enough
to consider torsors under tori. It was pointed out in \cite{Sko99}
that the theory works more generally under the assumption $\kbar[X]^\times=\kbar^\times$. This assumption is satisfied
when $X$ is proper, but it often fails for
many non-proper varieties;
it also fails for many homogeneous spaces of algebraic groups.

Harari and Skorobogatov (\cite{HS13}) extended the descent theory of Colliot-Th\'el\`ene and
Sansuc to the general case of a
smooth and geometrically integral variety (without the assumption $\kbar[X]^\times=\kbar^\times$). Their main results are almost same with Colliot-Th\'el\`ene and
Sansuc's.  For a geometrically integral variety $X$, to prove that strong approximation with Brauer--Manin obstruction holds on $X$, by Harari and Skorobogatov's open descent theory (\cite{HS13}), it needs to prove that  strong approximation holds on its torsors. However, if $X$ has non-constant invertible regular functions, then its torsors also have  non-constant invertible regular functions, thus its torsors will not satisfy strong approximation. Because of this reason, it is hard to find an application of the open descent theory.
In this paper, We give a new version of the open descent theory, its proof essentially depends on \cite[Theorem 2]{Har08}, a similar result is also gotten by Yang Cao.
By the new version of the open descent theory, to prove that strong approximation with Brauer--Manin obstruction holds on $X$, it only  need to prove that it holds on its torsors.

In section 1, we mainly prove the following theorem,  some applications are also given.
\begin{theorem} Let $X$ be a smooth and geometrically integral variety over $k$ (do not require $\kbar[X]^\times=\kbar^\times$).  Let $S$ be a torus with a morphism $\chi: \widehat {S}\rightarrow KD'(X)$. Let $\mathcal T$ be a torus of $X$ with the extended type $\chi$. Then we have
$$X(\A_k)^{\Br_1(X)}=\bigcup_{f: \mathcal T\xrightarrow{\chi} X}f(\mathcal T(\A_k)^{\Br_1(\mathcal T)}).$$
\end{theorem}

In section 2, we give some more applications of the new version of the open descent theory. In particular, we prove the following result, which generalizes  the results for toric varieties (\cite{CX13}) and for more general groupic varieties (\cite{CX15}).
\begin{theorem}  Let $k$ be a number field, $G$ a connected linear algebraic group and $H$ its connected subgroup. Let $X$ be a $G$-variety.
Suppose $X$ contains an open subset which is isomorphic to $G/H$ and the natural $G$-action on $G/H$ is compatible with the $G$-action on $X$.

Let $\Sigma$ be a finite set of places of $k$ containing all archimedean places. Assume $\prod_{v\in \Sigma}G'(k_v)$ is not compact for any non-trivial simple factor $G'$ of $G^{sc}$, and one of the following conditions holds:

1) $H$ is solvable;

2) $\Sigma$ contains at least one nonarchimedean place.\\
Then $X$ satisfies strong approximation with algebraic Brauer-Mannin obstruction off $\Sigma$.
\end{theorem}

\section{The proof of Theorem 0.1}
In this section, we mainly prove Theorem 0.1. First, we give some lemmas.
\begin{lemma}\label{lem:1} Let $C$ be a geometrically rational variety over $k$ and $K=\kbar(C)$. For any smooth geometrically integral variety $Y$, let $Y_K=Y\times_k K$.  Then $\Pic(\Ybar)\cong \Pic(Y_K)$.
\end{lemma}
\begin{proof} Let $K$ be a union of smooth $\kbar$-algebra of finite type over $\kbar$. For any such algebra $A$, $A $ is rational over $\kbar$, then we have $$\Pic(\Ybar \times_\kbar A)\cong \Pic(\Ybar)\oplus \Pic(A).$$
If one passes over to the limit over all such $A's$, one gets
$$\Pic(\Ybar)\cong \Pic(Y_K).\qedhere$$
\end{proof}

Let $S$ be a torus. Suppose $Y$ is a $S$-torsor of $X$. Let $m: S\times Y \rightarrow Y$ be the $S$-action of $Y$. Define (see \cite[(6.4.1)]{San81}) the morphism
$$\varphi: \Br_a(Y)\xrightarrow{m^*}\Br_a(S \times Y) \cong \Br_a(S)\times \Br_a(Y)\xrightarrow{\pi_1} \Br_a(S),$$
where $\pi_1$ is the first projection.

\begin{lemma}\label{lem:2} Let $Y$ be a torsor of $X$ under a torus $S$. Then we have the exact sequence
$$\Br_1(X)\rightarrow \Br_1(Y) \xrightarrow{\varphi} \Br_a(S).$$
\end{lemma}
\begin{proof} It follows from \cite[Corollary 2.10]{BD13}.
%
\end{proof}

For any geometrically integral variety $X$ over $k$, let $KD'(X)$ be (see \cite{HS13}) the following complex of
$Gal(\kbar/k)$-modules in degrees $-1$ and $0$:
$$[\kbar(X)^\times /\kbar^\times \to \text{Div}(\Xbar)].$$
Up to shift, $KD'(X)$ was introduced
by Borovoi and van Hamel in \cite{BH09} as $\text{UPic}(\Xbar)[1]$. By \cite[Proposition 6]{BH09}, one has $\Br_a(X)\cong H^1(k,KD'(X)).$

\begin{lemma} \label{lem:3} Let $C$ be a $S$-torsor over $k$ and $[C]=-\alpha \in H^1(k,S)$, $Y$ is a $S$-torsor over $X$, $Y^\alpha =: Y\times^S C$  is the twist of $Y$ by $\alpha$.
We have the commutative diagram
$$\xymatrix{
  \Br_a(Y) \ar[dr]_{\varphi} \ar[r]^{\cong}
                & \Br_a(Y^\alpha) \ar[d]^{\varphi_\alpha}  \\
                & \Br_a(S).             }
$$
\end{lemma}
\begin{proof}
Let $K:=\kbar(C)$. For any geometrically integral variety $Y$ over $k$, let $KD'(Y_K)$ be the following complex of
$\Gal(\kbar/k)$-modules in degrees $-1$ and $0$:
$$[K(Y)^\times/K^\times \to \text{Div}(Y_K)].$$
We have the morphism of the distinguished  triangles
\begin{equation}\label{ds:1}
\begin{CD}
    \kbar[Y^\alpha]^\times/\kbar^\times[-1] @>>> KD'(Y^\alpha)@>>> \Pic(\overline{Y^\alpha})\\
    @VV\cong V @VVV @VV\cong V\\
    K[Y^\alpha]^\times/K^\times[-1] @>>> KD'(Y^\alpha_K)@>>> \Pic(Y^\alpha_K).
  \end{CD}
\end{equation}
The first $'\cong'$ follows from \cite[Proposition 3.1]{CT07}, the last $'\cong'$ follows from Lemma \ref{lem:1}. Hence the second map is a quasi-isomorphism.

On the other hand, we have the commutative diagram of distinguished triangles
\begin{equation}\label{ds:2}
\begin{CD}
    K[Y^\alpha]^\times/K^\times[-1] @>>> KD'(Y^\alpha_K)@>>> \Pic(Y^\alpha_K)\\
    @AA\cong A @AA\cong A @AA\cong A\\
    K[Y]^\times/K^\times[-1] @>>> KD'(Y_K)@>>> \Pic(Y_K)\\
    @AA\cong A @AAA @AA\cong A\\
    \kbar[Y]^\times/\kbar^\times[-1] @>>> KD'(Y)@>>> \Pic(\Ybar).
  \end{CD}
\end{equation}
The last isomorphism follows from Lemma \ref{lem:1}.
Then the quasi-isomorphism $KD'(Y)\cong KD'(Y^\alpha)$  follows from the commutative diagram (\ref{ds:1}) and (\ref{ds:2}).
Then $\Br_a(Y) \cong \Br_a(Y^\alpha)$ by \cite[Corollary 2.20]{BH09}.

One has the following commutative diagram
\begin{equation}\label{ds:3}
\begin{CD}
    KD'(Y^\alpha) @>>> KD'(S\times Y^\alpha) @>>>KD'(S)\\
    @VV\cong V @VV\cong V @VV\cong V\\
    KD'(Y^\alpha_K)@>>> KD'(S_K\times Y^\alpha_K)@>>>KD'(S_K)\\
    @AA\cong A @AA\cong A @AA\cong A\\
    KD'(Y)@>>> KD'(S\times Y)@>>>KD'(S),
  \end{CD}
\end{equation}
the morphisms $\varphi_\alpha, \varphi$ are respectively induced by the composite maps of the first and third rows, then the proof follows.
\end{proof}

%
%
%

\begin{lemma} \label{lem:4} Let $X$ and $Y$ be smooth geometrically integral varieties over $k$, and $Y$ geometrically rational, then $\Br_a(X\times Y)\cong \Br_a(X) \times \Br_a(Y)$.
\end{lemma}
\begin{proof}
By \cite[Lemma 5.1]{BH09}, one has $KD'(X\times Y)\cong KD'(X)\times KD'(Y)$, the proof follows from \cite[Proposition 6]{BH09}.
\end{proof}

\begin{lemma}\label{lem:5} Let $G$ be a connected linear algebraic group over $k$. Then

i) The multiplication $m: G\times G \rightarrow G$ induces a map $$m^*: \Br_a(G)\rightarrow \Br_a(G)\times \Br_a(G)$$ such that $m^*(b)=(b,b)\in \Br_a(G)\times \Br_a(G).$

ii) The map $\tau:G\times G\times G \rightarrow G\times G$ is given by $\tau(g_1,g_2,g_3)=(g_1g_3,g_2g_3)$. Then it  induces the map $$\tau^*: \Br_a(G)\times \Br_a(G) \rightarrow \Br_a(G)\times \Br_a(G)\times\Br_a(G)$$ such that $\tau^*(b_1,b_2)=(b_1,b_2,b_1+b_2).$
\end{lemma}
\begin{proof} Let $i_1: G\rightarrow G\times G$ be given by $i_1(g)=(g,1_G)$, where $1_G$ is the identity element of $G$. Then we can see $id_G=m\circ i_1$. Therefore $$Id=i_1^* \circ m^*.$$ Choose $b\in \Br_a(G)$, suppose $m^*(b)=(x,y)$, then we have $$b=i_1^*(x,y).$$
Let $p_1,p_2$ be the two projections of $G\times G$ to $G$. Then $id_G=p_1\circ i_1$ and $0=p_2\circ i_1$. Hence $Id=i_1^* \circ p_1^*$ and $0=i_1^* \circ p_2^*$. Therefore $i_1^*(x,0)=x$ and $i_1^*(0,y)=0$, $i.e.$, $i_1^*(x,y)=x$. Therefore $b=i_1^*(x,y)=x$. Similarly, we replace $i_1$ by  $i_2: G\rightarrow G\times G$ which is given by $i_2(g)=(1_G,g)$, we will have $b=i_2^*(x,y)$ and $i_2^*(x,y)=y$, hence $b=y$. Therefore $m^*(b)=(b,b)$, we complete the proof of part $i)$.

We can see $\tau$ is the composite map
$$(G\times G)\times G \xrightarrow{id\times \Delta} (G\times G)^2\xrightarrow{m} G\times G,$$
where $\Delta: G\rightarrow G\times G$ is the diagonal map and $m(g_1,g_2,g_1',g_2')=(g_1g_1',g_2g_2')$.
Let $p_1,p_2$ be the two projections of $G\times G$ to $G$, we have $id_G=p_1\circ \Delta =p_2\circ \Delta$. Therefore we have $Id=\Delta^*\circ p_1^*=\Delta^*\circ p_2^*$. Hence $x=\Delta^*( p_1^*(x))=\Delta^*(x,0)$ and $y=\Delta^*( p_2^*(y))=\Delta^*(0,y)$. Therefore $$\Delta^*(x,y)=x+y.$$
Since $\tau=m\circ (id,\Delta)$, one has
$$\tau^*=(id^*,\Delta^*)\circ m^*.$$
By $i)$, we have  $m^*(b_1,b_2)=(b_1,b_2,b_1,b_2)$. Hence $$\tau^*(g_1,g_2)=(id^*,\Delta^*)(b_1,b_2,b_1,b_2)=(b_1,b_2,b_1+b_2).\qedhere$$
\end{proof}

Let $\eta$ is the generic point of $X$ and $L=k(X)$, then $Y_\eta$ is a principal homogeneous space of $S_L$. Let $\Br_1(Y_\eta):=\ker[\Br(Y_\eta)\rightarrow \Br(Y_\eta\times_L\overline L)]$, and $\Br_a(Y_\eta)$ be the quotient of $\Br(L)\to \Br_1(Y_\eta)$. Then we have the canonical isomorphism
\begin{equation}\label{eq:eta2}
\Br_a(S) \cong \Br_a(Y_\eta).
\end{equation}
The natural morphism $$\Br_1(Y)\to \ker[\Br(Y_\eta)\rightarrow \Br_a(Y_\eta\times_L\overline L)]$$ gives the following natural morphism
\begin{equation}\label{eq:eta}
\Br_a(Y)\rightarrow \Br_a(Y_\eta)\cong\Br_a(S).
\end{equation}

\begin{lemma}\label{lem:6} Let $Y$ be a $S$-torsor of $X$. the map $\varphi: \Br_a(Y)\rightarrow \Br_a(S)$ is coincide with the natural map in (\ref{eq:eta}).
\end{lemma}
\begin{proof}
Now we consider the image of $\varphi$.
We have the commutative diagram
 \begin{equation*}
\begin{CD}
    \Br_a(Y) @>m^*>> \Br_a(S \times Y) @>\cong>> \Br_a(S)\times \Br_a(Y) @>>> \Br_a(S)\\
    @VV V @VVV @VVV @VVV\\
   \Br_a(Y_\eta) @>m^*>> \Br_a(S_L \times Y_\eta) @>\cong>> \Br_a(S_L)\times \Br_a(Y_\eta)@>>> \Br_a(S_L),
  \end{CD}
\end{equation*}
Applying (\ref{eq:eta2}) and (\ref{eq:eta}), the above commutative diagram can be written as the following
 \begin{equation*}
\begin{CD}
    \Br_a(Y) @>m^*>> \Br_a(S \times Y) @>\cong>> \Br_a(S)\times \Br_a(Y) @>>> \Br_a(S)\\
    @VV V @VVV @VV V @VVV\\
   \Br_a(S) @>m^*>> \Br_a(S \times S) @>\cong>> \Br_a(S)\times \Br_a(S)@>>> \Br_a(S).
  \end{CD}
\end{equation*}
By part i) of Lemma \ref{lem:5}, the composite map of the second row is $\varphi$ and the composite map of the second row is the identity map.
The proof follows from the above commutative diagram.
\end{proof}

\begin{theorem}\label{thm:1}
 Let $X$ be a smooth and geometrically integral variety over $k$ (do not require $\kbar[X]^\times=\kbar^\times$).  Let $S$ be a torus with a morphism $\chi: \widehat {S}\rightarrow KD'(X)$. Let $\mathcal T$ be a torus of $X$ with the extended type $\chi$. Then we have
$$X(\A_k)^{\Br_1(X)}=\bigcup_{f: \mathcal T\xrightarrow{\chi} X}f(\mathcal T(\A_k)^{\Br_1(\mathcal T)}).$$
\end{theorem}

\begin{proof} $"\supset"$ is obvious, so we only need to show $"\subset"$.

Let $(p_v)_v\in X(\A_k)^{\Br_1(X)}$, by the open descent (\cite{HS13}), there is a $(q_v)\in \mathcal T(\A_k)$ for some $\mathcal T$, such that $f(q_v)=p_v$, for all $v$, and for all $\mathcal A\in \Br_1(X)$,
$$\sum_vf^*(\mathcal A)(q_v)=\sum_v \mathcal A(p_v)=0.$$

By Lemma \ref{lem:2}, we have the exact sequence $$\Br_1(X)\rightarrow \Br_1(\T) \xrightarrow{\varphi} \Br_a(S).$$
Let $\mathcal B \in \varphi^{-1}(\Sh^2(\widehat S))$, for any $(q_v)_v\in \T(\A_k)$, $\mathcal B(q_v)$ is a constant which does not depend on the choice of $q_v$, hence $\sum_v \mathcal B(q_v)$ is also a constant.  If  $\varphi(\mathcal B)=0$, then $\mathcal B$ is the image of an element $\AA$ of $\Br_1(X)$, hence $$\sum_v \mathcal B(q_v)=\sum_v\AA(p_v)=0.$$
Therefore, for any $\mathcal C\in Ima(\varphi)\cap \Sh^2(\widehat S)$, choose $\mathcal B\in \Br_1(\T)$ such that $\varphi(\mathcal B)=\mathcal C$, the sum $\sum_v\mathcal B(q_v)$ does not depend on the choice of $(q_v)_v$ and $\mathcal B$ by the above argument. Then we define a morphism
$$\tilde \alpha: Ima(\varphi)\cap \Sh^2(\widehat S) \to \Q/\Z.$$

First we will Modify $\T$ by its twist, such that $\tilde \alpha$ is a zero map.

By Poito-Tate duality, we can choose $\alpha\in \Sh^1(S)\cong \Sh^2(\widehat S)^*$, such that the restriction of $\alpha$ on $Ima(\varphi)\cap \Sh^2(\widehat S)$ is just $\tilde \alpha$.
Let $\T^\alpha$ is the twist of $\T$ by $\alpha$, let $[C]=-\alpha\in H^1(k,S)$, then $\T^\alpha=\T\times^S C$. Since $\alpha\in \Sh^1(S)$, we can see $(p_v)_v$ is also the image of $\T^\alpha(\A_k)$.
By similar definition of $\tilde \alpha$, we can also define a morphism $$Ima(\Br_1(\T^\alpha))\cap \Sh^2(\widehat S) \to \Q/\Z.$$
In the following, we will show this map is zero.

We can see $\T\times C$ is a $S$-torsor of $\T^\alpha$.
Let $K=k(C)$, choose $c$ be a fixed point of $C(K)$, then we have the morphism $\T_K \rightarrow \T_K\times C_K$ given by $p\mapsto (p,c)$. Since $k(C)\cap \kbar=k$, this morphism induces a morphism $ \Br_a(\T)\times \Br_a(C) \cong \Br_a(\T\times C) \rightarrow \Br_a(\T) $ which is given by $(\mathcal B,\mathcal C) \mapsto \mathcal B$. Since the composite map  $\T_K \rightarrow \T_K\times C_K\rightarrow \T_K^\alpha$ is an isomorphism, it induces an isomorphism $\Br_a(\T^\alpha) \rightarrow \Br_a(\T)$ which is just the isomorphism in Lemma \ref{lem:3}.

By Lemma \ref{lem:2}, we have the exact sequence
\begin{equation} \label{exa:br}
\Br_a(\T^\alpha)\xrightarrow{\psi} \Br_a(\T\times C)\xrightarrow{\varphi} \Br_a(S).
\end{equation}
By Lemma \ref{lem:4}, we have the natural isomorphism $\Br_a(\T\times C)\cong \Br_a(\T)\times \Br_a(C)$. For any $\mathcal B\in \Br_a(\T^\alpha)$, we have $$\psi(\mathcal B)=(\mathcal B,\mathcal C)\in \Br_a(\T)\times \Br_a(C).$$

Let $\eta$ is the generic point of $X$ and $L=k(X)$, then $Y_\eta$ is a principal homogeneous space of $S_L$.
%
Now we consider the image of $\varphi$, note that $\T\times C$ is a $S$-torsor of $\T^\alpha$.
Let $L=k(X)$, we have the commutative diagram
 \begin{equation*}
\begin{CD}
    \Br_a(Y\times C) @>m^*>> \Br_a(S \times Y\times C) @>>> \Br_a(S)\times \Br_a(Y)\times \Br_a(C) @>>> \Br_a(S)\\
    @VV V @VVV @VVV @VVV\\
   \Br_a(Y_\eta\times C_L) @>m^*>> \Br_a(S_L \times Y_\eta\times C_L) @>>> \Br_a(S_L)\times \Br_a(Y_\eta)\times \Br_a(C_L)@>>> \Br_a(S_L),
  \end{CD}
\end{equation*}
Since $C$ is a $S$-torsor, one has
\begin{equation}\label{eq:eta1}
\Br_a(C_L)\cong \Br_a(S)\cong \Br_a(S_L).
\end{equation}
By the two isomorphisms (\ref{eq:eta2}) and (\ref{eq:eta1}), the above commutative diagram  gives the commutative diagram
 \begin{equation*}
\begin{CD}
    \Br_a(Y\times C) @>m^*>> \Br_a(S \times Y\times C) @>>> \Br_a(S)\times \Br_a(Y)\times Br_a(C) @>>> \Br_a(S)\\
    @VV V @VVV @VV V @VVV\\
   \Br_a(S\times S) @>m^*>> \Br_a(S \times S\times S) @>>> \Br_a(S)\times \Br_a(S)\times \Br_a(S)@>>> \Br_a(S),
  \end{CD}
\end{equation*}
Let $\tilde {\mathcal B}$ be the image of $\mathcal B$ in $\Br_a(S)$ by the map (\ref{eq:eta2}). By Lemma \ref{lem:6}, we know $\tilde {\mathcal B}=\varphi(\mathcal B)$.
By the above commutative diagram and the part ii) in lemma \ref{lem:5}, we have
$$\varphi(\mathcal B,\mathcal C)=\tilde {\mathcal B}+\mathcal C.$$
Then we have $\psi(\mathcal B)=(\mathcal B, -\tilde {\mathcal B})$ by the exact sequence (\ref{exa:br}).

For any $(q_v,c_v)_v\in \T(\A_k)\times C(\A_k)$, let $(q_v')_v$ be its image in $\T^\alpha(\A_k)$ and the image of $(q_v')_v$ in $X(\A_k)$ is $(p_v)_v$. For any $\mathcal B\in \varphi^{-1}( \Sh^2(\widehat S))\subset\Br(\T^\alpha)$, its image is $(\mathcal B,-\tilde {\mathcal B})$ in $\Br_a(Y)\times \Sh^2(\widehat S)$. Then we have
$$\sum_v\mathcal B(q_v')=\sum_v(\mathcal B(q_v)-\tilde {\mathcal B}(c_v)).$$
By \cite[Lemma 8.4]{San81}, we have $$\sum_v\tilde {\mathcal B}(c_v)=-<-\alpha,\mathcal B>_{PT}=<\alpha,\mathcal B>_{PT}=\sum_v \mathcal B(q_v),$$
the last equation is from the choice of $\alpha$. Therefore we have $\sum_v\mathcal B(q_v')=0$.

Now we replace $\T$ by $\T^\alpha$, then for any $\mathcal B\in \varphi^{-1}(\Sh^2(\widehat S))$, we have
$\sum_v\mathcal B(q_v)=0$ for any $q_v\in \T(\A_k)$ and the image of $(q_v)_v$ is $(p_v)_v$.

Let $\tilde {\mathcal B}\in \Br_a(S)$ and $\mathcal B \in \Br_a(\T)$ with $\varphi(\mathcal B)=\tilde {\mathcal B}$.
The point $(q_v)_v$ give a morphism $\rho: Ima(\varphi)\to \Q/\Z$ by $\tilde {\mathcal B} \mapsto \sum_v\mathcal B(q_v)$. By the above argument, we know the definition of $\rho$ is well defined and $\rho |_{Ima(\varphi)\cap \Sh^2(\widehat S)}=0$.
We can choose $\tilde \rho: \Br_a(S)\rightarrow \Q/\Z$ such that $\tilde {\rho}|_{Ima(\varphi)}=\rho$ and $\tilde {\rho} |_{\Sh^2(\widehat S)}=0$.

By \cite[Theorem 2]{Har08}, we have the exact sequence
$$ S(\A_k)\rightarrow \Br_a(S)^D \rightarrow \Sh^2(\widehat S)^D \rightarrow 0.$$
Therefore, we can choose $(s_v)_v\in S(\A_k)$, such that $-\tilde {\rho}$ is the image of $(t_v)_v$ by the first map of the above exact sequence, $i.e.$, $$\sum_v\tilde {\mathcal B}(s_v)=-\sum_v \mathcal B(q_v)$$ for any $\mathcal B \in \Br(\T)$. Then we replace $(q_v)_v$ by $(q_v.s_v)_v$, its image is also $(p_v)_v$.
 For any $\mathcal B \in \Br(\T)$, the multiplication $m: S\times \T\rightarrow \T$ induced the map $m^*: \Br_a(\T)\rightarrow \Br_a(S)\times \Br_a(\T)$ by $\mathcal B \mapsto (\tilde {\mathcal B},\mathcal B)$ by the definition of $\varphi$, then we have
$$\sum_v\mathcal B(q_v.s_v)=\sum_v \mathcal B(q_v)+\sum_v \tilde {\mathcal B}(q_v)=0,$$
thus we complete the proof.
\end{proof}

The following result is immediately a consequence of Theorem \ref{thm:1}.
\begin{corollary}\label{cor:desc} Let $S_1\subset \Omega_k$ be a finite set.  If strong approximation with algebraic Brauer-Manin obstruction off $S_1$ holds on the torsors of $X$ with the extended type $\chi$, then it also holds on $X$.
\end{corollary}

\begin{example} Let $K$ be finite etale $k$-algebra of dimension $n$ and $\omega_1,\cdots, \omega_n$ its $k$-basis. Let $m\geq n-1$ and let $Z$ be the closed subset of the projective space $\Bbb P^m$ defined by $N_{K/k}(x_1w_1+\cdots+x_nw_n)=0$. Then $\Bbb P^m\setminus Z$ satisfies strong approximation with algebraic Brauer-Manin obstruction off $\infty_k$.
\end{example}
\begin{proof} Use the embedding $\Bbb P^m\setminus Z\subset \Bbb P^m$, there is a morphism $$\chi: \Pic(\Bbb P^m_\kbar)\cong KD'(\Bbb P^m)\rightarrow KD'(X).$$
 Any torsor of $\Bbb P^m\setminus Z$ with the extended type $\chi$ is the restriction of a universal torsor of $\Bbb P^m$ on $\Bbb P^m\setminus Z$. However, the universal torsor of $\overline {\Bbb P^m}$ is unique, it is just $\Bbb A^{m+1}$. Then torsors of $P^m\setminus Z$ with the extended type $\chi$ $\T$ is just $R_{K/k}(\G_{m,K})\times \Bbb A^{m-n}$, obviously it satisfies strong approximation with algebraic Brauer-Manin obstruction  off $\infty_k$ (see \cite{Har08}). By Corollary \ref{cor:desc}, $\Bbb P^m\setminus Z$ satisfies strong approximation with algebraic Brauer-Manin obstruction off $\infty_k$.
\end{proof}

\section{Applications}

In this section, we give some application of Theorem \ref{thm:1}. The following theorem had already been proved in \cite{CX13}.
\begin{theorem} \label{thm:toric} Let $X$ be a smooth toric variety over $k$.
Then $X$ satisfies strong approximation with algebraic Brauer--Manin obstruction off $\infty_k$.
\end{theorem}
 \begin{proof} We have an embedding $X\to X^c$, where $X^c$ is a smooth proper toric variety. Then there is a morphism $$\chi: \Pic(\overline {X^c})\cong KD'(X^c)\rightarrow KD'(X).$$
 Since any torsor of $X$ with the extended type $\chi$ is the restriction of a universal torsor of $\overline {X^c}$ on $X$, let $\T$ be a such torsor, it is a restriction of a universal torsor $\T'$ of $\overline {X^c}$.

 Let $T\subset X$ be the torus contained in $X$ as an open subset.  Let $\text{Div}_{\Tbar}(\overline{X^c})$ be the $\Tbar$-invariant Weil divisor of $\overline{X^c}$. By \cite[Theorem 4.1.3]{Cox}, we have the following exact sequence
$$0\to \That(\cong \kbar[T]^\times/\kbar^\times) \xrightarrow{\text{div}} \text{Div}_{\Tbar}(\overline{X^c}) \to \Pic(\overline{X^c}) \to 0.$$
Let $\T'$ be a universal torsor of $X^c$. Let $\Mhat :=\text{Div}_{\Tbar}(\overline{X^c})$ and $M$ the dual torus of $\Mhat$. Then $$\Mhat=\Mhat_1\times \Mhat_2,$$
where $\Mhat_1=\text{Div}_{\Tbar}(\overline{X}), \Mhat_2=\text{Div}_{\overline{X^c}\setminus \Xbar}(\overline{X^c})$. We now want to apply \cite[Theorem 2.3.1, Corollary 2.3.4]{CTS87} for the
local description of universal torsors over $X^c$.
The restriction $\T'_T$ of the universal torsor $\T$ to $T$ has the form $\T_T  = M \times_T T\cong M$.

%
%


Let $\That$ be the group of characters of $T$. Let $N:=X_*(T)=\Hom(\That,\Z)$, $\Delta$ a fan in $N_\Bbb R:=N\otimes \Bbb R$ and $X=X_\Delta$ the toric variety associated to $\Delta$.
Denote $\Deltabar$ to be the fan of $\Delta$ omitting the $\Gamma_k$-action.
We shall call a 1-dimensional cone a ray. Let $\Deltabar(1)$ be the set of rays of $\Deltabar$. Let $D_\rho$ be the $\Tbar$-invariant Weil divisor of $X$ associated to $\rho\in \Deltabar(1)$. Then $\Mhat_1=\sum_{\rho\in \Deltabar(1)}\Z D_\rho$. Let $\tilde N:=X_*(M)=\Hom(\Mhat,\Z)$, ${\tilde D}_\rho\in \tilde N$ the dual of $D_\rho$. We can see $\{{\tilde D}_\rho: \rho\in \Deltabar(1)\}$ is a subbasis of $\tilde N$.

Let $C$ be the cone in $\tilde N$ generated by $\{{\tilde D}_\rho:\rho\in \Deltabar(1)\}$. Let $ Y_C$ be the toric variety of $C\subset \tilde N$ with the natural $M$-action. The permutation module $\Mhat_1$ is a factor of $\Mhat$, then we can see $Y_C=Y\times_k M_2$, where $Y$ is an affine space.


Let $R_\rho\subset \tilde N$ be the ray generated by ${\tilde D}_\rho$. Then $\{R_\rho: \rho\in \Deltabar(1)\}$  are all rays of $C$. We choose $$C':=\{0\}\cup \bigcup_{\rho\in \Deltabar(1)}\{R_\rho\}\subset C,$$
$C'$ is a subfan of $C$. Let $Y_{C'}$ be the toric variety associated to $C'$.
The variety $Y_{C'}=Y\setminus Z\times_k M_2$ and $Z$ has codimension $2$.
Let $f: \tilde N\to N$ be the morphism induced by $\text{div}: \That\to \Mhat $.
We claim that there is a toric morphism $Y_{C'} \to X_{\Deltabar}$, $i.e.$, there is a morphism of fans $C'\to \Deltabar$ which is induced by $f:\tilde N\to N$.  

For any $\rho\in \Deltabar(1)$, let $u_\rho\in N$ be the minimal generator of $\rho$. By \cite[Proposition 4.1.2]{Cox}, we have $$\text{div}: \That \to \Mhat, \chi \mapsto \text{div}(\chi)=\sum_{\rho\in \Deltabar}\chi(u_\rho)D_\rho.$$
For $\sigma\in \Deltabar(1)$, $f({\tilde D}_\sigma)\in N (=\Hom_\Z(\That, \Z))$ and
\begin{equation*}
f({\tilde D}_\sigma)=\left(\chi\mapsto {\tilde D}_\sigma(\sum_{\rho\in \Deltabar(1)}<u_\rho,\chi> D_\rho)=<u_\sigma,\chi>\right),
\end{equation*}
then $f({\tilde D}_\sigma)=u_\sigma$.
Therefore $f(R_\rho)=\rho$ for any $\rho\in \Deltabar(1)$, hence $f$ induces a morphism of fans $C' \to \overline {\Delta'}$.
%

Recall that The variety $Y_{C'}=Y\setminus Z\times_k M_2$ and $Z$ has codimension $2$.
By Lemma \cite[Lemma 1.1]{Wei14}, $Y\setminus Z$ satisfies strong approximation off $\infty_k$, hence $Y_{C'}$ satisfies strong approximation with algebraic Brauer--Manin obstruction off $\infty_k$.

For any $(p_v) \in \T(\A_k)^{\Br_1(\T)}$, by Theorem 0.1, there is a torsor $f: \T \to X$ and $(r_v) \in \T(\A_k)^{\Br(\T)}$ with the extended type $\chi$ such that $(f(r_v)) = (p_v)$.
By \cite[Lemma 3.4]{CX15}, the torsor $\T$ is in fact of a $M$-toric variety. Let $T_1$ is the dual torus of $\kbar[\T]^\times/\kbar^\times$, in fact $\T_1\cong M_2$. Let $g:\T \to T_1$ be the natural morphism of $M$-varieties which extends the group morphism $M\to T_1$, then any geometrical fiber of $g$ is smooth and geometrical integral.
Therefore we can choose $(r_v') \in M(\A_k)$ such that $r_v'$ is very closed to $r_v$ in $\T(\A_k)$ and $g(r_v')=g(r_v)$. One has $\Br_1(T_1)\cong\Br_1(\T)$ since $\Pic(\overline\T)=0$, hence $g(r_v')=g(r_v)$ induces $\AA(r_v')=\AA(r_v)$ for any $\AA\in \Br_1(\T)$, $i.e.$, $(r_v')\in M(\A_k)^{\Br_1(\T)}$. Since $\Br_1(\T)\cong\Br_1(Y_{C'})$, one has $(r_v')\in M(\A_k)^{\Br_1(Y_{C'})}$. Since $Y_{C'}$ satisfies strong approximation with algebraic Brauer--Manin obstruction off $\infty_k$, we can choose a point $r\in Y_{C'}(k)$ which is very closed to $(r_v')_{v\nmid \infty}$ in $Y_{C'}(\A^\infty_k)$, thus its image in $X(k)$ is very closed to $(p_v)_{v\nmid \infty}$ in $X(\A^\infty_k)$.
\end{proof}

 For any connected linear algebraic group $G$, the reductive part $G^{red}$ of $G$ is given by $$ 1\rightarrow R_u(G)\rightarrow G \rightarrow G^{red} \rightarrow 1 $$ where $R_u(G)$ is the unipotent radical of $G$. Let $G^{ss}=[G^{red}, G^{red}]$ be the semi-simple part of $G$, let $G^{sc}$ be the semi-simple simply connected covering of $G^{ss}$.
\begin{theorem} \label{thm:general}  Let $k$ be a number field, $G$ a connected linear algebraic group and $H$ its connected subgroup. Let $X$ be a $G$-variety.
Suppose $X$ contains an open subset which is isomorphic to $G/H$ and the natural $G$-action on $G/H$ is compatible with the $G$-action on $X$.

Let $\Sigma$ be a finite set of places of $k$ containing all archimedean places. Assume $\prod_{v\in \Sigma}G'(k_v)$ is not compact for any non-trivial simple factor $G'$ of $G^{sc}$, and one of the following conditions holds:

1) $H$ is solvable;

2) $\Sigma$ contains at least one nonarchimedean place.\\
Then $X$ satisfies strong approximation with algebraic Brauer-Mannin obstruction off $\Sigma$.
\end{theorem}

\begin{lemma}\label{lemma:2-1}
Let $G$ be a connected linear algebraic group and $H$ its connected subgroup. Suppose $\Pic(\overline G/\overline H)=0$. Let $T$ be the dual torus of $\kbar[G/H]^\times/\kbar^\times$. Let $G\rightarrow G/H \to T$
 be the natural morphism of  $G$-varieties and the composite map is a group morphism. Then any geometric fiber of $G/H \to T$ only has constant invertible regular functions and trivial Picard group. In particular, the algebraic Brauer group of the fiber over $T(k)$ is trivial.
\end{lemma}
\begin{proof} Let $G_1$ be the kernel of $ G \rightarrow T$. Any geometric fiber of $f:G/H \to T$ is isomorphic to $G_1/H$ over $\kbar$. Then we only need to show $\kbar[G_1/H]^\times =\kbar^\times$ and $\Pic(\overline{G_1}/\overline{H})=0$.

Let $K:=\kbar(G)$. Let $\eta: Spec(K)\to T$ be the generic point of $T$, the generic fiber $f_\eta$ has a $K$-point and it is a homogeneous space of $G_1\times_k K$, its stabilizer is $H\times_k K$. Hence we have $f_\eta \cong_K G_1/H \times_k K$. By \cite[Proposition 3.1]{CT07}, one has $$K[f_\eta]^\times/K^\times \cong \kbar[G_1/H]^\times/\kbar^\times.$$
All fibers of $G/H \to T$ are geometrically integral, by \cite[Proposition 3.1 and 3.2]{CT07}, one has the exact sequence
$$0\to \widehat T \xrightarrow{\cong} \kbar[G/H]^\times/\kbar^\times \to \kbar[G_1/H]^\times/\kbar^\times \to \Pic(\overline T)=0,$$
hence $\kbar[G_1/H]^\times =\kbar^\times.$

By (\cite{San81}, proof of Corollary 6.11, p. 44), the exact sequence $$0\rightarrow G_1\rightarrow G\rightarrow T\rightarrow 0$$ gives the exact sequence
$$\widehat{G}\to\widehat{G_1}\to \Pic(\overline T)\rightarrow \Pic(\overline G)\rightarrow \Pic(\overline {G_1})\rightarrow 0,$$
hence $\Pic(\overline G) \cong \Pic(\overline {G_1})$ and $\widehat{G}\to\widehat{G_1}$ is surjective.

By \cite[Propsition 6.10]{San81}, we have the commutative diagram of exact sequences
 \begin{equation*}
\begin{CD}
    \widehat{G_1} @>>>  \widehat{H} @>>> \Pic(\overline{G_1}/\overline{H}) @>>> \Pic(\overline {G_1})@>>> \Pic(\overline H)\\
     @AAA @| @AAA @AA\cong A @|\\
  \widehat{G} @>>>  \widehat{H}@>>> \Pic(\overline{G}/\overline{H}) @>>>\Pic(\overline G) @>>> \Pic(\overline H),
  \end{CD}
\end{equation*}
hence we have $\Pic(\overline{G_1}/\overline{H})\cong\Pic(\overline G/\overline{H})=0$.
\end{proof}

\begin{proof}[Proof of Theorem 2.2]

Let $X^c$ be a smooth compactification of $X$, then $\Pic(\overline{X^c})$ is a free module. We have a morphism $$\Pic(\overline {X^c})\cong KD'(X^c)\rightarrow KD'(X).$$
Let $\chi$ be the above morphism $\Pic(\overline {X^c})\rightarrow KD'(X)$ and let $\T$ be a torsor $\T$ of $X$ with the extended type $\chi$.

The composite map $$G\rightarrow G/H \hookrightarrow X$$
gives a torsor $G\times_X \T$ of $G$ under $S$, where $S$ is the dual torus of $\Pic(\overline {X^c})$. Replace $\T$ by its twist under an element of $H^1(k,S)$, we can assume $G\times_X \T$ has a $k$-point.
By \cite[Theorem 5.6]{CT08}, $G\times_X \T$ is a connected linear algebraic group denoted by $\tilde{G}$. There is the exact sequence of algebraic groups
\begin{equation}\label{exa:g}
1\rightarrow S\rightarrow \tilde{G}\rightarrow G\rightarrow 1.
\end{equation}

By a similar argument in \cite[Lemma 3.4]{CX15}, the $G$-action of $X$ can be naturally uniquely extended to a $\tilde{G}$-action of $\T$. The variety $\T\times_X (G/H)$ is a homogeneous space of $\tilde{G}$ and has a $k$-point, its stabilizer is denoted by $\tilde H$. There is the exact sequence $$0\rightarrow S \rightarrow \tilde H\rightarrow H\rightarrow 0.$$ In particular, if $H$ is solvable,  then $\tilde H$ is solvable, hence $\Pic(\overline{\tilde H})=0$.

For any $(s_v) \in X(\A_k)^{\Br_1(X)}$, by Theorem \ref{thm:1}, there is $\alpha\in H^1(k,S)$ and $(s_v')\in \T^{\alpha}(\A_k)^{\Br_1(\T^{\alpha})}$ such that $f(s_v')=s_v$, where $f: \T^{\alpha}\to X$. Obviously $f^{-1}(G/H)$ is a homogenous space of $\tilde G$, then $\T^{\alpha}(\A_k)^{\Br_1(\T^{\alpha})}\neq \emptyset$ implies $\T^{\alpha}(k)\neq \emptyset$ by \cite[Theorem 2.2]{Bo96}.  In fact we also have $f^{-1}(G/H)(k) \neq \emptyset$, then $f^{-1}(G/H) \cong \tilde G/\tilde H$, hence $\T^{\alpha}$ is also a $\tilde G$-variety and $\tilde G/\tilde H\subset \T^{\alpha}$. Now we denote $\T^\alpha$ also by $\T$. The proof follows from the claim that $\T$ satisfies strong approximation with algebraic Brauer--Manin obstruction off $\Sigma$. In the following, we will proof the claim.

By the choice of $\chi$, the torsor $\T$ of $X$ is the restriction of a universal torsor $\T'$ of $X^c$ on $X$. By \cite[Proposition 2.1.1]{CTS87}, we have $\kbar[\T']^\times\cong \kbar^\times$ and $\Pic(\T')=0$, hence
\begin{equation}\label{Pic}
\Pic(\T)=\Pic(\overline {\tilde G}/\overline {\tilde H})=0.
\end{equation}



The exact sequence (\ref{exa:g}) gives the exact sequence $$1\rightarrow S\rightarrow \tilde{G}^{red}\rightarrow G^{red}\rightarrow 1.$$
Let $M$ be the kernel $\tilde{G}^{tor}\rightarrow G^{tor}$, and let $\mu$ be the kernel of the induced map $S\to M$. There is  the commutative diagram of exact sequences
$$  \begin{CD} @. 1 @. 1 @. 1\\
 @. @VVV  @VVV @VVV\\
1 @>>> \mu @>>> \tilde{G}^{ss} @>>> G^{ss} @>>> 1 \\
@.    @VVV  @VVV  @VVV\\
1 @>>> S @>>> \tilde{G}^{red} @>>> G^{red} @>>> 1 \\
@. @VVV @VVV @VVV \\
1 @>>> M @>>> \tilde{G}^{tor} @>>> G^{tor} @>>> 1 \\
@. @VVV @VVV @VVV \\
@. 1 @. 1 @. 1
\end{CD} $$
The exact sequence (\ref{exa:g}) gives an exact sequence
$$0\to \widehat{G^{tor}}\to \widehat{\tilde G^{tor}}\to \widehat{S}\to \Pic(\overline G).$$
Since $\Pic(\overline G)$ is finite (see \cite[Proposition 4.14]{CT07}), the kernel of $S \to \tilde G^{tor}$ is finite, hence $\mu$ is finite.
Then the first row of the above commutative diagram induces
\begin{equation}\label{eq:sc}
\tilde{G}^{sc}=G^{sc}.
\end{equation}

Denote $Y:=\tilde G/\tilde H\subset \T$. Let $T_0$ and $T_1$ be respectively the dual tori of $\kbar[Y]^\times/\kbar^\times$ and $\kbar[\T]^\times/\kbar^\times$. There is an exact sequence
$$0 \to \widehat{T_1} \to \widehat{T_0}\to \Div_{\overline \T\setminus \overline Y}(\overline\T)\to 0.$$
Let $\tilde N:=X_*(T_0)=\Hom(\widehat{T_0},\Z)$. Let $D_1,\cdots, D_n$ be the prime divisors in $\overline\T\setminus \Ybar$.  Let ${\tilde D_i} \in \tilde N$ be the dual of $D_i$ and $\tilde R_i\subset \tilde N$ the ray generated by ${\tilde D_i}$. Then $C:=\{\tilde R_i:i=1,\cdots,n\}$ is a fan in $\tilde N$. Let $V$ be the toric variety associated to $C$. By Theorem \ref{thm:toric}, $V$ satisfies strong approximation with algebraic Brauer--Manin obstruction off $\Sigma$.

By a similar argument as in the proof of Proposition 3.12 in \cite{CX15}, there exists an open $\tilde G$-subvariety $U\subset \T$ such that $\tilde G/\tilde H \subset U$,
$\text{codim}(\T\setminus U,\T)\geq 2$ and the canonical map $\tilde G/\tilde H\rightarrow T_0$ can be extended a morphism $U\xrightarrow{\varsigma_U} V$ of $\tilde G$-varieties. Moreover, the morphism $\varsigma_{U}$  is smooth with nonempty and geometrically integral fibres.

Now we use the fibration $U\rightarrow V$, note $V$ satisfies strong approximation with algebraic Brauer--Manin obstruction off $\Sigma$.
For any $(p_v) \in U(\A_k)^{\Br_1(U)}$, let $(r_v)\in V(\A_k)^{\Br_1(V)}$ be its image, by an easy fibration argument similar as in \cite[Proposition 3.1]{CTX13}, we can find a point $r\in T_0(k)$ same as in the proof of \cite[Proposition 3.1]{CTX13}. In particular the point $r$ is very closed to $(r_v)_{v\nmid \infty}\in V(\A^\Sigma)$ such that there is a point $(p_v') \in U_r(\A_k)$ such that $(p_v')_{v\not \in \Sigma}$ is very closed to $(p_v)_{v\not \in \Sigma}$ in $ U(\A^\Sigma)$.

We consider the composite morphism of $\tilde G \rightarrow\tilde G/\tilde H\rightarrow T_0$, its kernel is a connected linear algebraic group denoted by $\tilde G_1$. For any point in $T_0(k)$, its fiber of $\tilde G/\tilde H\rightarrow T_0$ is a homogeneous space of $\tilde G_1$ with the stabilizer $\overline{\tilde H}$, it 
(hence $U_r$) has trivial algebraic Brauer group by (\ref{Pic}) and Lemma \ref{lemma:2-1}. Therefore $U_r$ has a rational point by \cite[Theorem 2.2]{Bo96} and contains an open subset which is isomorphic to $\tilde G_1/\tilde H$.

Since $\tilde G_1$ is the kernel of $G$ to a torus, one has $\tilde G_1^{ss}=\tilde G^{ss}$, hence $\tilde G_1^{sc}=\tilde G^{sc}$. By (\ref{eq:sc}), one has $\tilde G_1^{sc}= G^{sc}$, thus $\prod_{v\in \Sigma}\tilde G_1'(k_v)$ is not compact for any non-trivial simple factor $\tilde G_1'$ of $\tilde G_1^{sc}$  since the same property holds on $G$.

If $H$ is solvable, then $\tilde H$ is also solvable, hence $\Pic(\overline{\tilde H})=0$. By \cite[Propsition 2.10]{San81}, we have the exact sequence $$\Pic(\overline{\tilde H})\to \Br(\overline {\tilde G_1}/\overline {\tilde H})\to \Br(\overline {\tilde G_1}).$$
Therefore the morphism $\Br(\overline {\tilde G_1}/\overline {\tilde H})\to \Br(\overline {\tilde G_1})$ is injective. Then we have
\begin{equation}\label{br}
\Br_a(\tilde G_1/\tilde H,\tilde G_1)=\Br_a(\tilde G_1/\tilde H)=0
\end{equation}
 by (\ref{Pic}) and Lemma \ref{lemma:2-1}, where $\Br_a(\tilde G_1/\tilde H,\tilde G_1)=\ker[\Br(\tilde G_1/\tilde H)\to \Br(\overline {\tilde G_1})]$.

By (\ref{br}) and \cite[Theorem 0.1]{BD13} for the case $H$ is solvable and  Lemma \ref{lemma:2-1} and \cite[Theorem 0.2]{BD13} for the other case, $\tilde G_1/\tilde H$ (hence $U_r$) satisfies strong approximation off $\Sigma$. Thus we can find a point $p\in U_r(k)$ which is very closed to $(p_v')_{v\not \in \Sigma}$ in $ U_r(\A^\Sigma)$, then $p$ is also very closed to $(p_v)_{v\not \in \Sigma}$ in $U(\A^\Sigma)$.
So $U$ satisfies strong approximation with algebraic Brauer--Manin obstruction off $\Sigma$.

Recall that $T_1$ is the dual torus of $\kbar[\T]^\times/\kbar^\times$. Let $g:\T \to T_1$ be the natural morphism of $\tilde G$-varieties, then all geometrical fibers of $g$ are isomorphic since $T_1$ just has one $\tilde G$-orbit, hence they are smooth and geometrical integral.
For any $(q_v) \in \T(\A_k)^{\Br_1(\T)}$ in $\T(\A_k)$, then we can choose $(q_v') \in U(\A_k)$ such that $(q_v') $  is very closed to $(q_v) $ in $\T(\A_k)$ and $g(q_v')=g(q_v)$. Note that $g(q_v')=g(q_v)$ induces $\AA(q_v')=\AA(q_v)$ for any $\AA\in Ima[\Br_1(T_1) \to \Br_1(\T)]$. Since $\text{codim}(\T\setminus U,\T)\geq 2$, one has
$$\Br_1(T_1) \xrightarrow{\cong}\Br_1(\T) \xrightarrow{\cong}\Br_1(U).$$
Therefore, we have $(q_v')\in U(\A_k)^{\Br_1(U)}$. Since $U$ satisfies strong approximation with algebraic Brauer--Manin obstruction off $\Sigma$, we can choose a point $q\in U(k)$ which is very closed to $(q_v')_{v\not \in \Sigma}$ in $ U(\A^\Sigma)$, hence $q$ is also very closed to $(q_v)_{v\not \in \Sigma}$ in $\T(\A^\Sigma)$.
Then $\T$ satisfies strong approximation with algebraic Brauer--Manin obstruction off $\Sigma$.
\end{proof}

%

\bf{Acknowledgment} 	
 \it{The work is supported by National
  Key Basic Research Program of China (Grant No. 2013CB834202) and
  National Natural Science Foundation of China (Grant Nos. 11371210 and
  11321101). }



\bigskip
{\small

%
%

{\scshape
D. Wei: Academy of Mathematics and System Science,  CAS, Beijing
100190, P.R.China}
\smallskip

{\it E-mail: }
\url{dshwei@amss.ac.cn}
}

\end{document}